\documentclass[letterpaper, 10 pt, conference, onecolumn]{ieeeconf}

\usepackage[utf8]{inputenc}    
\usepackage{amssymb}
\usepackage{amsfonts}
\usepackage{amsmath}
{
      \newtheorem{assumption}{Assumption}
  }
\usepackage{graphicx}
\usepackage{mathtools}
\usepackage{xfrac}
\usepackage{nicefrac}
\usepackage{mathrsfs}
\usepackage{breqn}
\usepackage{dsfont}
\usepackage{tabularx}
\usepackage{relsize}
\usepackage{times}
\usepackage{url}
\usepackage{adjustbox}
\usepackage{boldline}
\usepackage{algorithm}
\usepackage[T1]{fontenc}
\usepackage[dvips]{epsfig}
\usepackage{epstopdf}
\usepackage{color}
\DeclareGraphicsExtensions{.ps,.eps,.pdf}

\newfloat{procedure}{htbp}{loa}
\floatname{procedure}{Procedure}

\IEEEoverridecommandlockouts
\newtheorem{theorem}{Theorem}
\newtheorem{corollary}{Corollary}
\newtheorem{remark}{Remark}

\newtheorem{problem}{Problem}

\begin{document}
\date{}

\title{Guaranteed simulation error bounds\\for linear time invariant systems identified from data}

\author{Marco~Lauricella  and 
	Lorenzo~Fagiano
	\thanks{The authors are with the Dipartimento di Elettronica, Informazione e Bioingegneria, Politecnico di Milano, Piazza Leonardo da Vinci
32, 20133 Milano, Italy. E-mail addresses: \{marco.lauricella $|$ lorenzo.fagiano \}@polimi.it. Corresponding author: Marco Lauricella}}

\maketitle

\begin{abstract}
\centering This is a technical report that extends and clarifies the results presented in \cite{LF_CDC18}.
\end{abstract}

\section{Problem formulation}
\label{s:probl_form}
Consider a discrete time, asymptotically stable, strictly proper linear time invariant system, with input $u(k) \in \mathbb{R}$ and output $z(k) \in \mathbb{R}$, where $k \in \mathbb{Z}$ is the discrete time variable. The state-space representation of the system dynamics is given by:
\begin{equation}
\label{eq:sys_ss_def}
\begin{aligned}
x(k+1)&=A x(k)+Bu(k) \\
z(k)&=Cx(k)
\end{aligned}
\end{equation}
where $x \in \mathbb{R}^n$ is the system state. The output measurement $y(k)$ is affected by an additive disturbance $d(k)$:
\begin{equation} 
\label{eq:output_def}
y(k)=z(k)+d(k),
\end{equation}
\begin{assumption}
\label{as:1} (Disturbance and input bounds)
\begin{itemize} 
\item $|d(k)| \leq \bar{d}_0, \; \forall k \in \mathbb{Z}$.
\item $u(k) \in \mathbb{U} \subset \mathbb{R}, \; \forall k \in \mathbb{Z}$, $\mathbb{U}$ compact.\hfill \QED
\end{itemize}
\end{assumption}
\begin{assumption}\label{A:model_reach_obsv} (Observability and reachability)  
	The system at hand is completely observable and reachable. \hfill \QED
\end{assumption}
\vspace{0.2cm}
Assumption \ref{A:model_reach_obsv} is made for simplicity, as it can be relaxed by considering only  the observable and controllable sub-space of the system state. 
For a given value of $k$ and of prediction horizon $p\in\mathbb{N}$, we have $z(k+p)=CA^p x(k) + C\sum\limits_{i=1}^{p} g_i u(k+p-i)$, where $g_i=A^{i-1} B$. Under Assumption \ref{A:model_reach_obsv}, this can be equivalently written as:
\begin{equation}
\label{eq:real_sys_ARX_model}
z(k+p)=\underbrace{[Z_n^T(k) \; U_{p,n}^T(k)]}\limits_{\psi_p(k)^T}
\underbrace{\left[\begin{array}{c} \theta_{p,z}^0\\ \theta_{p,u}^0 \end{array}\right]}\limits_{\theta_p^0}=\psi_p(k)^T\theta_p^0,
\end{equation}
where $^T$ denotes the matrix transpose operation, and:
\begin{equation}\label{eq:regressor_true}
Z_n(k) \doteq \begin{bmatrix} z(k) \\ z(k-1) \\ \vdots \\ z(k-n+1) \end{bmatrix}, \; U_{p,n}(k) \doteq \begin{bmatrix} u(k+p-1) \\ u(k+p-2) \\ \vdots \\ u(k-n+1) \end{bmatrix}.
\end{equation}
It is well-known that, for an asymptotically stable system, the parameters $\theta_p^{0,(i)}$ are subject to the following bounds:
\begin{equation}
\label{eq:real_decay_rate}
\begin{array}{rcl}
|\theta_{p,u}^{0,(i)}| &\leq& L_u\rho^i,\;i=1,\ldots,p+n-1\\
|\theta_{p,z}^{0,(i)}| &\leq& L_z\rho^{p+i},\;i=1,\ldots,n-1\\
\end{array},
\end{equation}
where $^{(i)}$ denotes the element in the $i$-th position of a vector. In \eqref{eq:real_decay_rate}, the decay rate $\rho$  and the constants $L_u$ and $L_z$ depend on the system matrices in \eqref{eq:sys_ss_def}; in particular, $\rho$ is dictated by the magnitude of the system's dominant poles. Finally, we can write the one-step-ahead dynamics of the true system as (considering $p=1$ in \eqref{eq:real_sys_ARX_model}):
\begin{equation}
\label{eq:real_sys_ARX_1s_model}
z(k+1)=\psi_1(k)^T \theta_1^0,
\end{equation}
which corresponds to a standard auto-regressive description with exogenous input (ARX). For any $p>1$, the entries of the parameter vector $\theta_p^0$ are polynomial functions of the entries of $\theta_1^0$, readily obtained by recursion of \eqref{eq:real_sys_ARX_1s_model}. We indicate this polynomial dependency in compact form as:
\begin{equation}
\label{eq:polynomials}
\theta_p^0=h(\theta_1^0,p,n).
\end{equation}

As motivated in \cite{LF_CDC18}, we consider the problem of identifying the parameters of a one-step-ahead model of \eqref{eq:real_sys_ARX_1s_model} from data. To this end, we introduce the model regressor $\varphi_p(k) \in \mathbb{R}^{2o+p-1}$, where  $o \in \mathbb{N}$ is the chosen model order:
\begin{equation} \label{eq:regr_def}
\varphi_p(k) \doteq [Y_o^T(k) \; U_{p,o}^T(k)]^T,
\end{equation}
where $Y_o(k) \doteq [ y(k) \, y(k-1) \ldots y(k-o+1)]^T\in\mathbb{R}^{o}$ and $U_{p,o}$ is defined as in \eqref{eq:regressor_true}. Then, we consider the following ARX model structure for our one-step-ahead model:
\begin{equation} \label{eq:1s_pred_model_def}
\hat{z}(k+1)=\varphi_1(k)^T \theta_1,
\end{equation}
where $\hat{z}(k+1)\approx z(k+1)$ is the predicted one-step-ahead output, and $\theta_1 \in \mathbb{R}^{2o}$ is the model parameter vector to be estimated from data.  Simulating (i.e. iterating) the model \eqref{eq:1s_pred_model_def} defines the following multi-step predictors for each $p>1,\,p\in\mathbb{N}$:
\begin{equation} \label{eq:pred_model_def}
\hat{z}(k+p)=\varphi_p(k)^T \begin{bmatrix} \hat{\theta}_{p,y} \\ \hat{\theta}_{p,u} \end{bmatrix} =\varphi_p(k)^T \hat{\theta}_p,
\end{equation}
where $\hat{z}(k+p)$ is the predicted (i.e. simulated) $p$-step ahead future output, and $\hat{\theta}_p = h(\hat{\theta}_1,p,o) \in \mathbb{R}^{2o+p-1}$ is the corresponding parameter vector, whose entries are polynomial functions of the entries of $\hat{\theta}_1$. 

Besides the possible  order mismatch (i.e. $o\neq n$), the main difference between the model \eqref{eq:1s_pred_model_def} and the true system \eqref{eq:real_sys_ARX_1s_model} is that the former employs  disturbance-affected measurements $y(k)$ of the output in its regressor, instead of the true output values $z(k)$. To study the effects of this difference, let us define the vector $\psi_{p,o}(k)\doteq[Z_o^T(k) \; U_{p,o}^T(k)]^T$, where $Z_o$ is obtained as in \eqref{eq:regressor_true}. Assumption \ref{as:1}, along with the asymptotic stability of the  system, implies that the regressors $\psi_{p,o}(k)$ belong to a compact set $\Psi_{p,o}$:
\begin{equation}\label{eq:set_regr_true_o}
\psi_{p,o}(k) \in \Psi_{p,o} \subset \mathbb{R}^{2o+p-1}, \, \Psi_{p,o} \text{ compact}, \, \forall p \in \mathbb{N}, \, \forall k \in \mathbb{Z}.
\end{equation}
Consequently, $\varphi_p(k)$ belongs to a compact set $\Phi_p$ as well:
\begin{equation}\label{eq:set_regr_noise_o}
\varphi_p(k) \in \Phi_p= \Psi_{p,o}\oplus\mathbb{D}_p,\; \forall p \in \mathbb{N}, \; \forall k \in \mathbb{Z},
\end{equation}
where $F\oplus M=\{f+m:f\in F,\,m\in M\}$ is the Minkowski sum of two given sets $F,M$, and
\begin{equation}\label{eq:set_noise:o}
\mathbb{D}_p\doteq\{[d^{(1)},\ldots,d^{(o)},0,\ldots,0]^T:|d^{(i)}|\leq\bar{d}_0\}\subset\mathbb{R}^{2o+p-1}
\end{equation} 
is the set of all possible disturbance realizations that can affect the system output values stacked inside the regressor $\varphi_p$. In practical applications, the sets $\Psi_{p,o}$ and $\Phi_p$ depend on the input/output trajectories of the system, and they are typically not available explicitly. However, for the sake of parameter identification we assume to have a finite number $N$ of measured pairs $(\tilde{\varphi}_p(i),\tilde{y}_p(i))$, where $\tilde{\cdot}$ denotes a specific sample and $\tilde{y}_p(i)\doteq\tilde{y}(i+p)$. These sampled data define the set:
\begin{equation}
\label{eq:sample_var_set}
\tilde{\mathscr{V}}_p^{N} \doteq \left\{ \tilde{v}_p(i)= \begin{bmatrix} \tilde{\varphi}_p(i) \\ \tilde{y}_p(i) \end{bmatrix}, \; i=1,\hdots,N \right\} \subset \mathbb{R}^{2o+p},
\end{equation}
The continuous counterpart of $\tilde{\mathscr{V}}_p^{N}$ is:
\begin{equation}
\label{eq:all_var_set}
\mathscr{V}_p \doteq \left\{ v_p= \begin{bmatrix} \varphi_p \\ y_p \end{bmatrix}: y_p \in Y_p(\varphi_p), \; \forall \varphi_p \in \Phi_p \right\} \subset \mathbb{R}^{2o+p},
\end{equation}
where $Y_p(\varphi_p) \subset \mathbb{R}$ is the compact set of all possible measured  output values corresponding to every value of $\varphi_p \in \Phi_p$ and every disturbance realization $d:|d|\leq\bar{d}_0$. 
\begin{assumption} (Informative content of data)
\label{as:set_distance} 
For any $\beta > 0$, there exists a value of $N< \infty$ such that:
$$ d_2 \left( \mathscr{V}_p, \tilde{\mathscr{V}}_p^{N} \right) \leq \beta,$$
where $ d_2 \left( \mathscr{V}_p, \tilde{\mathscr{V}}_p^{N} \right) \doteq \underset{v_1 \in \mathscr{V}_p}{\textrm{max}} \underset{v_2 \in \tilde{\mathscr{V}}_p^{N}}{\textrm{min}} \left\Vert v_2-v_1 \right\Vert_2 $ represents the distance between the two sets. \hfill \QED
\end{assumption}
The meaning of Assumption \ref{as:set_distance} is that, by adding more points to the measured data-set, the set of all the trajectories of interest is densely covered, leading to $ \underset{N \to \infty}{\textrm{lim}} d_2 \left( \mathscr{V}_p, \tilde{\mathscr{V}}_p^{N} \right)=0$. This corresponds to a persistence of excitation condition, plus a bound-exploring property of the variable $d(k)$.

We can now state the problem addressed in this paper.
\begin{problem}
\label{p:probl_statement}
Under Assumptions \ref{as:1}-\ref{as:set_distance}, use the available data \eqref{eq:sample_var_set} to:
\begin{itemize}
	\item[a)] estimate the disturbance bound $\bar{d}_0$, the system order $n$, and the decay rate $\rho$;
	\item[b)] identify the parameters of the model \eqref{eq:1s_pred_model_def} according to a suitable optimality criterion, together with associated guaranteed bounds on the simulation (i.e. multi-step prediction) error $|z(k+p)-\hat{z}(k+p)|,\,p=1,\ldots,\overline{p}$, where $\bar{p}<\infty$ is a maximum simulation horizon of interest. \hfill \QED
\end{itemize}
\end{problem}
We provide next an approach, based on multi-step Set Membership (SM) identification, to address point a) of Problem \ref{p:probl_statement}, and to obtain worst-case bounds useful to solve point b) as well.
\section{Multi-step Set Membership identification of linear systems}
\label{s:MS_models_def}
\subsection{Preliminary results}\label{SS:SM_preliminary}
We start by recalling results derived in \cite{TFFS}, which we employ and complement with further ones in the next sections. Consider a generic $p\in\mathbb{N}$ and a generic parameter vector $\theta_p$ defining a multi-step predictor $\varphi_p(k)^T\theta_p\approx z(k+p)$ (not necessarily computed by iterating a one-step-ahead model). By denoting the error between the system output and such an estimate as $\varepsilon_p(\theta_p,\varphi_p(k))=z(k+p)-\varphi_p(k)^T \theta_p$,
under Assumption \ref{as:1} it follows that:
\begin{equation}
\label{eq:total_est_err}
\left\vert y(k+p) - \varphi_p(k)^T \theta_p \right\vert \leq \bar{\varepsilon}_p(\theta_p)+\bar{d},
\end{equation}
where $\bar{\varepsilon}_p(\theta_p)$ represents the global error bound produced by  $\theta_p$ (termed ``global'' since it holds for all possible regressor values in the set $\Phi_p$), and $\bar{d} \geq 0$ is an estimate of the true disturbance bound $\bar{d}_0$. $\bar{\varepsilon}_p(\theta_p)$ is given by:
\begin{equation}
\label{eq:epsilon_def}
\begin{aligned}
\bar{\varepsilon}_p(\theta_p) = &\min_{\varepsilon \in \mathbb{R}} \; \varepsilon \;\;\text{subject to}\\
&\, \left\vert y_p - \varphi_p^T \theta_p \right\vert \leq \varepsilon + \bar{d}, \; \forall (\varphi_p,y_p): \begin{bmatrix} \varphi_p \\ y_p \end{bmatrix} \in \mathscr{V}_p
\end{aligned}
\end{equation}
This bound cannot be computed exactly in practice, with a finite set of data points. In \cite{TFFS}, a method for estimating $\bar{\varepsilon}_p(\theta_p)$ is proposed, along with the proof that this estimate, denoted with $\underline{\lambda}_p$, converges to $\bar{\varepsilon}_p$ from below under suitable assumptions. $\underline{\lambda}_p$ is obtained by solving the following linear program (LP):
\begin{equation}
\label{eq:lambda_est_def}
\begin{aligned}
\underline{\lambda}_p = &\min_{\theta_p,\lambda\geq 0} \; \lambda \; \; \text{subject to} \\
&\, \left\vert \tilde{y}_p - \tilde{\varphi}_p^T \theta_p \right\vert \leq \lambda + \bar{d}, \; \forall (\tilde{\varphi}_p,\tilde{y}_p): \begin{bmatrix} \tilde{\varphi}_p \\ \tilde{y}_p \end{bmatrix} \in \tilde{\mathscr{V}}_p^{N}
\end{aligned}
\end{equation}
Then, the estimate is inflated to account for the uncertainty due to the use of a finite number of measurements, leading to:
\begin{equation}
\label{eq:epsilon_hat_def}
\hat{\bar{\varepsilon}}_p = \alpha \underline{\lambda}_p, \; \alpha > 1.
\end{equation}
We can now recall the Feasible Parameter Set (FPS) $\Theta_p$, which is the tightest set of parameter values that are consistent with the information coming from data and disturbance bound estimate:
\begin{equation}
\label{eq:FPS_orig_def}
\Theta_p= \left\{ \theta_p : | \tilde{y}_p - \tilde{\varphi}_p^T \theta_p | \leq \hat{\bar{\varepsilon}}_p + \bar{d}, \; \forall (\tilde{\varphi}_p, \tilde{y}_p): \begin{bmatrix} \tilde{\varphi}_p \\ \tilde{y}_p \end{bmatrix} \in \tilde{\mathscr{V}}_p^{N} \right\} 
\end{equation}
If the FPS is bounded, it results in a polytope with at most $N$ faces (if it is unbounded, then the employed data are not informative enough and new data should be collected). Now, the FPS can be used to derive a global bound on the prediction error produced by a given value of $\theta_p$, indicated with $\tau_p(\theta_p)$:
\begin{equation}
\label{eq:tau_real_orig_def}
\begin{array}{l}
|z(k+p)-\hat{z}(k+p)|\leq\tau_p(\theta_p)\\
\tau_p(\theta_p) = \max\limits_{\varphi_p \in \Phi_p} \; \max\limits_{\theta \in \Theta_p} \; | \varphi_p^T (\theta - \theta_p)  | + \hat{\bar{\varepsilon}}_p.
\end{array}
\end{equation}
Similarly to $\bar{\varepsilon}_p$, also $\tau_p(\theta_p)$ cannot be computed exactly with a finite data set. An estimate is given by:
\begin{equation}
\label{eq:tau_est_orig_def}
\underline{\tau}_p(\theta_p) = \max_{\tilde{\varphi}_p \in \tilde{\mathscr{V}}_p^{N}} \; \max_{\theta \in \Theta_p} \; | \varphi_p^T (\theta - \theta_p)  | + \hat{\bar{\varepsilon}}_p.
\end{equation}
$\underline{\tau}_p(\theta_p)$ converges to its counterpart $\tau_p(\theta_p)$ from below as $N$ increases under Assumption \ref{as:set_distance}, see \cite{TFFS}. In practical applications, we inflate $\underline{\tau}_p(\theta_p)$ as well, in order to compensate for the uncertainty deriving from the usage of a finite data-set:
\begin{equation}
\label{eq:tau_hat_orig_def}
\hat{\tau}_p(\theta_p)=\gamma \left( \max_{\tilde{\varphi}_p \in \tilde{\mathscr{V}}_p^{N}} \; \max_{\theta \in \Theta_p} \; \left\vert \varphi_p^T (\theta - \theta_p)  \right\vert \right) + \hat{\bar{\varepsilon}}_p, \; \gamma > 1.
\end{equation}
\begin{assumption} (Estimated error bounds)
\label{as:eps_tau_overbound}
The estimated values of $\hat{\bar{\varepsilon}}_p$ and $\hat{\tau}_p(\theta_p)$ are larger than the corresponding true bounds $\bar{\varepsilon}_p$ and $\tau_p(\theta_p)$, respectively.  \hfill \QED
\end{assumption}
\begin{remark} (On the choice of $\alpha$ and $\gamma$)
\label{rm:choice_alpha_gamma}
The parameter $\alpha$ can be chosen sufficiently close to 1 if $N$ is big enough to `guarantee' that the experiment performed on the system is informative enough. A value of $\alpha$ that is too high will lead to a conservative error bound and larger FPSs, reducing the performance of the estimate. Similarly, with a large enough value of $N$, $\gamma$ can be chosen really close to 1 and still satisfy Assumption \ref{as:eps_tau_overbound}. An excessive value of $\gamma$ will produce a conservative error bound $\tau$, which could be far from the real performance achieved by the identified model. In a sense, $\alpha$ and $\gamma$ express how much one is confident on the informative content of the identification experiment. \hfill \QED
\end{remark} 
\subsection{New results on the estimated multi-step error bounds}
\label{s:exp_decay_estim}
We present two results showing additional properties of the quantity $\underline{\lambda}_p$ \eqref{eq:lambda_est_def}. These provide a theoretical justification to the estimation procedures for the disturbance bound $\bar{d}_0$, system order $n$, and decay rate $\rho$, which we propose in  Section \ref{ss:dim_lambda_dbar}.

Let us define:
\begin{equation}
\label{eq:lambda_per_proof_first_def}
\lambda_p \doteq \min_{\theta_p \in \Omega} \; \max_{ \left[ \begin{smallmatrix} \varphi_p \\ y_p \end{smallmatrix} \right] \in \mathscr{V}_p} \left( \left\vert y_p - \varphi_p^T \theta_p \right\vert - \bar{d} \right).
\end{equation}
In \eqref{eq:lambda_per_proof_first_def}, $\Omega \subset \mathbb{R}^{2o+p-1}$ represents a compact approximation of the real set $\mathbb{R}^{2o+p-1}$: it can be chosen e.g. by considering box constraints of $\pm 10^{15}$ on each element of the parameter vector. This is a technical assumption that allows us to use the maximum and minimum operators, instead of supremum and infimum. 
\begin{assumption} (Predictor order)
\label{as:model_order}
The estimated order $o$ is chosen such that $o \geq n$. \hfill \QED
\end{assumption} 
As indicated in Section \ref{ss:dim_lambda_dbar}, this assumption can be satisfied by initially over-estimating the system order, since the results presented below are not affected by the chosen value of $o$, as long as it is larger than $n$.
\begin{remark}
\label{r:models_order_matching}
With a slight abuse of notation, in the remainder we imply that, when $o \neq n$, the parameter vectors $\theta_p$ (if $o<n$), or $\theta_p^0$ (if $o>n$), are appropriately padded with zero entries to equate their dimensions, thus keeping consistency of all matrix operations. \hfill \QED
\end{remark}
\begin{theorem}
\label{th:lambda_d}
Consider the asymptotically stable system \eqref{eq:sys_ss_def}. If Assumptions \ref{as:1}-\ref{as:set_distance} and \ref{as:model_order} hold, then: 
\begin{enumerate}
\item $\lambda_p \xrightarrow{p\to\infty} (\bar{d}_0 - \bar{d})$
\item $\underline{\lambda}_p \leq \lambda_p$
\item $\forall \eta \in (0,\lambda_p], \; \exists N < \infty : \underline{\lambda}_p \geq \lambda_p-\eta$ \\
\end{enumerate} 
\end{theorem}
\begin{proof}
See the appendix.
\end{proof}
%
%
\begin{corollary}
\label{th:Delta_theta}
Consider the asymptotically stable system \eqref{eq:sys_ss_def}. If Assumptions \ref{as:1}-\ref{as:set_distance} and \ref{as:model_order} hold, and if the disturbance bound is correctly chosen as $\bar{d}=\bar{d}_0$, then:
\begin{equation}
\label{eq:lambda_true_no_dist}\lambda_p=\bar{d}_0\|\theta_{p,z}^0\|_1\leq n\,\bar{d}_0\,L_z\,\rho^{p+1} 
\end{equation} 
\end{corollary} 
\begin{proof}
	See the appendix.
\end{proof}
\begin{remark}
\label{r:results_consequences} Theorem \ref{th:lambda_d} and Corollary \ref{th:Delta_theta} imply two consequences that are useful for model identification. The first is that, when $\bar{d}=\bar{d}_0$ and $o<n$, $\lambda_p$ converges to a non-zero value as $p$ increases, which is due to the model order mismatch. The rationale behind this statement is that, when $o<n$, there exists a choice of $\varphi_p$ and $y_p$ inside $\mathscr{V}_p$ such that it is not possible to find a $\theta_p$ able to bring the error $\varphi_p^T (\theta_p^0 - \theta_p)$ to zero. This observation will be used to estimate the model order in the next section.  The second consequence is that $\lambda_p \xrightarrow{p \to \infty} 0$ with the same decay rate as that of the true system parameters, thus providing a way to estimate the latter. \hfill \QED
\end{remark}
\begin{remark}
\label{rm:prop_lambda_real}
Here we resort to the result demonstrated in \cite{TFFS}, which provides us with guarantees that the estimated $\underline{\lambda}_p$ converges to $\lambda_p$ from below as $N$ grows, meaning that also $\underline{\lambda}_p$ will undergo by the properties described by Theorem \ref{th:lambda_d} and Corollary \ref{th:Delta_theta}. \hfill \QED
\end{remark}
	
\subsection{Estimation of disturbance bound, system order, and decay rate}
\label{ss:dim_lambda_dbar}
From Theorem \ref{th:lambda_d} it follows that, for $N \to \infty$ and $o \geq n$, picking a disturbance bound estimate $\bar{d}\geq \bar{d}_0$ results in $\underline{\lambda}_p$ converging to zero as $p$ increases; instead, choosing $\bar{d} < \bar{d}_0$ results in $\underline{\lambda}_p$ converging to a non-zero value. We resort to this property to estimate the value of the disturbance bound, as described by Procedure \ref{p:d_bar_est_procedure}.
\begin{procedure}
\caption{Estimation of $\bar{d}_0$}
\label{p:d_bar_est_procedure}
\begin{enumerate}
\item Choose a large value as initial guess of $o$.
\item Set a starting value of $\bar{d}$ small enough to have $\bar{d}<\bar{d}_0$.
\item Gradually increase $\bar{d}$, recalculating each time $\underline{\lambda}_p$, until the first value of $\bar{d}$ under which $\exists \bar{p}: \underline{\lambda}_p=0 \; \forall p > \bar{p}$ is found.
\item The obtained $\bar{d}$ corresponds to the disturbance bound, and the related $\bar{p}$ represents the system settling time.
\end{enumerate}
\end{procedure} \\
Then, we propose an approach, based on the observation reported in Remark \ref{r:results_consequences}, to estimate the minimal model order that verifies Assumption \ref{as:model_order}, as described by Procedure \ref{p:o_est_procedure}.
\begin{procedure}
\caption{Estimation of $n$}
\label{p:o_est_procedure}
\begin{enumerate}
\item Set $\bar{d}$ and $\bar{p}$ to the values resulting from Procedure \ref{p:d_bar_est_procedure}.
\item Choose a large value as initial guess of $o$.
\item Gradually decrease $o$, recalculating each time $\underline{\lambda}_p$, until the first value of $o$ under which $\exists p>\bar{p} : \underline{\lambda}_p>0$ is found.
\item The last value of $o$ under which $\underline{\lambda}_p=0 \; \forall p > \bar{p}$ will be the minimal predictor order.
\end{enumerate}
\end{procedure}

Finally, the observed decay rate of $\underline{\lambda}_p$ can be used to estimate the exponentially decaying trends \eqref{eq:real_decay_rate} of the system. In particular, our goal is to derive quantities  $\hat{\rho}\approx\rho$, $\hat{L}_z\approx L_z$, and $\hat{L}_u\approx L_u$.

Let us define $\boldsymbol{f}_{\varepsilon} \doteq \left[\hat{\bar{\varepsilon}}_1 \; \cdots \; \hat{\bar{\varepsilon}}_{p_{\textrm{max}}} \right]^T$, where $\hat{\bar{\varepsilon}}_p$ is obtained from \eqref{eq:epsilon_hat_def} with $\bar{d}$ resulting from Procedure \ref{p:d_bar_est_procedure}, and $p_{\textrm{max}}>\bar{p}$. Let us also define, for given values of $\hat{L}$ and $\hat{\rho}$, the quantities $g_{L\rho}(p) \doteq \hat{L} \hat{\rho}^p,\,p\in[1,p_{\textrm{max}}]$. Then, we solve the following optimization problem to compute $\hat{\rho}$:
\begin{equation}
\label{eq:exp_dec_estim}
\begin{aligned}
\left[ \hat{L}, \hat{\rho} \right] =&\arg \; \min_{L, \rho} \; \left\Vert \boldsymbol{f}_{\varepsilon} - \boldsymbol{g}_{L \rho} \right\Vert^2_2   \\
&\;\text{subject to} \\
&\; \boldsymbol{g}_{L \rho} \succeq \boldsymbol{f}_{\varepsilon} \\
&\; L > 0, \; 0 < \rho < 1
\end{aligned}
\end{equation}
where $\boldsymbol{g}_{L\rho}= \begin{bmatrix} g_{L\rho}(1) & \cdots & g_{L\rho}(p_{\textrm{max}}) \end{bmatrix}^T$. In practice, the computed value of $\hat{\rho}$ minimizes the quadratic norm of the difference between $\hat{\bar{\varepsilon}}_p$ (i.e. the observed decay rate) and $g_{L\rho}(p)$ (the theoretical exponential decay rate). Supported by Corollary \ref{th:Delta_theta}, this estimate of $\hat{\rho}$ is consistent with the system decay rate. However, we still need to estimate suitable values of $\hat{L}_z,\,\hat{L}_u$. 
For the former, we exploit the FPSs $\Theta_p$ considering the parameters pertaining to the output values inside the regressors $\varphi_p$:
\begin{equation}
\label{eq:L_z_estimate}\hat{L}_z=\left. \left( \max_{p \in [1, \bar{p}]} \; \max_{\theta_p \in \Theta_p} \; \max_{i=1, \hdots, o} \; \theta_p^{(i)} \right) \middle/ \hat{\rho} \right. .
\end{equation}
Regarding $\hat{L}_u$, we instead consider the parameters pertaining to the $o$ most recent input values inside the regressors $\varphi_p$: 
\begin{equation}
\label{eq:L_final_est}
\hat{L}_u= \left. \left( \max_{p \in [1, \bar{p}]} \; \max_{\theta_p \in \Theta_p} \; \max_{i=o+1, \hdots, 2o} \; \theta_p^{(i)} \right) \middle/ \hat{\rho} \right. .
\end{equation}
Indeed, the magnitude of these parameters is not affected by the decay rate and it can be used to estimate the true bounds $L_z$ and $L_u$ (see \eqref{eq:real_decay_rate}).

\section{Identification of one-step-ahead predictors with guaranteed simulation error bounds}
\label{s:one_step_estimate}
Exploiting the results and procedures presented in Section \ref{s:MS_models_def}, we are now in position to address part b) of Problem \ref{p:probl_statement}. In particular, we present new methods to learn the parameters of one-step-ahead prediction models of the form \eqref{eq:1s_pred_model_def}, considering the simulation (multi-step) accuracy and trying to enforce asymptotic stability of the predictor as well. The first step is to refine the FPSs $\Theta_p$ \eqref{eq:FPS_orig_def}, by adding additional constraints that take into account the estimated system decay rate.
\subsection{Feasible Parameter Sets with constraints on the parameters decay rate}
\label{ss:new_FPS_def}
Let us define:
\begin{equation}
\label{eq:exp_bound_sets}
\Gamma_p= \bigg\{ \theta_p : |\theta_{p,u}^{(i)}| \leq \hat{L}_u \hat{\rho}^i, \; \forall i \in [1,p+o-1], \; \wedge \; | \theta_{p,y}^{(i)} | \leq \hat{L}_z \hat{\rho}^{p+i}, \; \forall i \in [1,o] \bigg\}
\end{equation}
Then, we modify the Feasible Parameter Sets as follows:
\begin{equation}
\label{eq:new_FPS_def}
\Theta_p^{L\rho}= \Theta_p \cap \Gamma_p.
\end{equation}
\begin{assumption} (Estimated decay rate)
	\label{as:Lr_overbound}
	The parameters of the estimated exponential decay rate are such that $\hat{\rho}\in[\rho,1)$, $\hat{L}_z\geq L_z$, and $\hat{L}_u\geq L_u$.  \hfill \QED
\end{assumption}
\begin{remark}
\label{r:FPS_non_empty}
Under Assumptions \ref{as:eps_tau_overbound}, \ref{as:model_order} and \ref{as:Lr_overbound}, it follows that $\theta_p^0 \in \Theta_p^{L\rho}, \; \forall p$, i.e. each FPS \eqref{eq:new_FPS_def} is non-empty and contains the parameters of the corresponding iterated model of the system \eqref{eq:real_sys_ARX_1s_model}. 
These assumptions cannot be verified in practice when a finite data-set is used. However, as long as the sets $\Theta_p^{L\rho}$ are non-empty (which can be easily verified, since they are all polytopes), we can be confident that the computed estimates and prior assumptions are not invalidated by data. Whenever $\Theta_p^{L\rho}$ becomes empty for some $p$, the estimated bounds can be enlarged until non-empty sets are obtained again. \hfill \QED
\end{remark}

We describe next two possible procedures to estimate $\theta_1$ \eqref{eq:1s_pred_model_def}, exploiting the modified FPSs. Both procedures are based on nonlinear programs.
\subsection{Method I - minimize the worst-case simulation error bound}
\label{ss:idea3}
This method is based on the concept of global error bound. Here we want to find the predictor model that minimizes the maximum worst-case error bound, along the considered prediction horizon. This is done by solving the following problem:
\begin{equation}
\label{eq:idea3}
\begin{aligned}
\hat{\theta}_1=&\arg \; \min_{\theta_1} \left\Vert \boldsymbol{\tau}\left( \boldsymbol{\theta} \right) \right\Vert_\infty \\
&\;\text{subject to} \\
&\; \theta_{p} \in \Theta_p^{L\rho}, \; \forall p \in [1,\bar{p}]
\end{aligned}
\end{equation}
where $\boldsymbol{\tau}\left( \boldsymbol{\theta} \right)= \begin{bmatrix} \hat{\tau}_1(\theta_1) & \hat{\tau}_2(\theta_2) & \cdots & \hat{\tau}_{\bar{p}}(\theta_{\bar{p}}) \end{bmatrix}^T$, $\theta_p=h(\theta_1,p,o)$, and $\hat{\tau}_p(\theta_p)$ is defined as in \eqref{eq:tau_hat_orig_def}. The resulting optimization problem is:
\begin{equation}
\label{eq:idea3_exp}
\begin{aligned}
\hat{\theta}_1=&\arg \min_{\theta_1} \left( \max_{p \in [1,\bar{p}]} \max_{\, i=1,\hdots,N} \max_{\, \theta \in \Theta_p^{L\rho}} \left\vert \tilde{\varphi}_p(i)^T (\theta-\theta_p) \right\vert + \hat{\bar{\varepsilon}}_p \right) \\
&\;\text{subject to} \\
&\; \theta_{p} \in \Theta_p^{L\rho}, \; \forall p \in [1,\bar{p}]
\end{aligned}
\end{equation}
Problem \eqref{eq:idea3_exp} can be rewritten into a simpler nonlinear minimization problem. First of all, the absolute value can be split in two terms introducing the following quantity: 
$$\check{\varphi}_p(j)= \left\{ \begin{matrix} \tilde{\varphi}_p(j) \quad \text{if} \quad  j \leq N \\ -\tilde{\varphi}_p(j) \quad \text{if} \quad j > N \end{matrix} \quad \text{for} \quad j=1,\hdots,2N. \right.$$
Then, by defining:
$$
c_{j_p}= \max_{\theta \in \Theta_p^{L\rho}} \; \check{\varphi}_p(j)^T \theta, \quad j=1,\hdots,2N, \quad p=1,\hdots,\bar{p},
$$
we can reformulate \eqref{eq:idea3_exp} as:
\begin{equation}
\label{eq:idea3_exp_ref1}
\begin{aligned}
\hat{\theta}_1=&\text{arg} \; \min_{\theta_1} \; \max_{p \in [1,\bar{p}]} \; \max_{i=1,\hdots,2N} (c_{j_p} - \check{\varphi}_p(j)^T \theta_p) \\
&\;\text{subject to} \\
&\; \theta_{p} \in \Theta_p^{L\rho}, \; \forall p \in [1,\bar{p}]
\end{aligned}
\end{equation}
The optimization problem defined by \eqref{eq:idea3_exp_ref1} corresponds to:
\begin{equation}
\label{eq:idea3_exp_ref2}
\begin{aligned}
\hat{\theta}_1=&\text{arg} \; \min_{\theta_1} \; \zeta \\
&\,\text{subject to} \\
&\, c_{j_p} - \check{\varphi}_p(j)^T \theta_p \leq \zeta, \; j=1,\hdots,2N, \; p=1,\hdots,\bar{p} \\
&\, \theta_{p} \in \Theta_p^{L\rho}, \; \forall p \in [1,\bar{p}]
\end{aligned}
\end{equation}
This results into a nonlinear optimization problem, having $2N$ linear constraints and $2N(\bar{p}-1)$ nonlinear constraints, plus $2N\bar{p}$ nonlinear constraints that requires the previous solution of $2N\bar{p}$ LP problems.
\subsection{Method II - minimize the simulation error enforcing the exponential decay rate}
\label{ss:idea4}
Here we propose a different approach, which is based on the SEM criterion. The idea is to minimize the simulation error produced by the one-step iterated prediction model, given a certain initial condition $\varphi_1(0)$. This results in:
\begin{equation}
\label{eq:idea4}
\begin{aligned}
\hat{\theta}_1=&\text{arg} \; \min_{\theta_1 \in \Theta_1^{L\rho}} \left\Vert \tilde{\boldsymbol{Y}} - \hat{\boldsymbol{Z}}(\theta_1) \right\Vert^2_2 \\
&\,\text{subject to} \\
&\, \theta_p \in \Gamma_p, \; \forall p \in [2,N]
\end{aligned}
\end{equation}
where $\tilde{\boldsymbol{Y}}= \begin{bmatrix} \tilde{y}(1) & \tilde{y}(2) & \cdots & \tilde{y}(N) \end{bmatrix}^T$, $\hat{\boldsymbol{Z}}(\theta_1)= \begin{bmatrix} \tilde{\varphi}_1(0)^T \theta_1 & \tilde{\varphi}_2(0)^T \theta_2 & \cdots & \tilde{\varphi}_N(0)^T \theta_N \end{bmatrix}^T$, and $\theta_p=h(\theta_1,p,o)$. 
Equation \eqref{eq:idea4} corresponds to a nonlinear optimization problem, having $2N$ linear constraints.

\section{Simulation results}
\label{s:sim_res}
The performance of the proposed identification approaches has been assessed through their application to a simulation case study. We resort to a SISO, asymptotically stable, underdamped system, whose output is affected by a uniformly distributed random noise, bounded in the interval $[-0.1, \; 0.1]$ (i.e. $\bar{d}_0=0.1$). The transfer function of said system is:
\begin{equation}
G(s) = \frac{160}{\left( s+10 \right) \left( s^2 + 0.8s + 16 \right)}
\end{equation}
Input and output data points are acquired with a sampling time $T_s=0.1$. The data-set collected for the identification phase and the data-set used for the validation phase contain $N=1500$ and $N_v=1500$ samples of each signal, respectively. The input signal takes values in the set $\{ -1 ; \; 0; \; 1\}$ randomly every $10$ time units. Fig. \ref{f:y_noise} depicts the behavior of the measured system output during the identification experiment.
%
\begin{figure}[thpb]
	\centering
	\includegraphics[width=0.6\columnwidth]{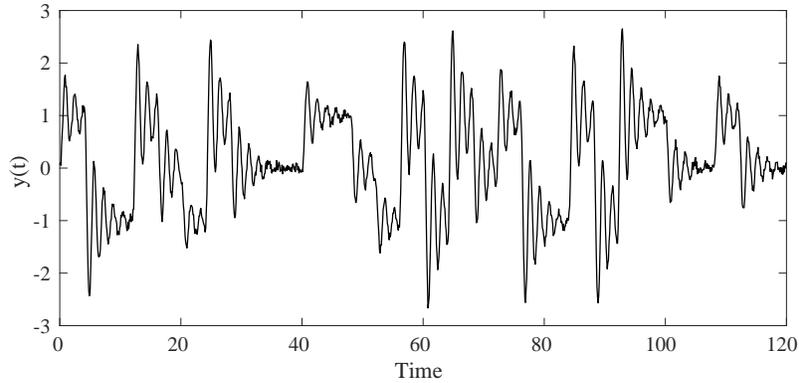}
	\caption{Measured system output during the identification experiment.}
	\label{f:y_noise}
\end{figure}
\begin{figure}[thpb]
	  \centering
      		\begin{tabular}{c}
     			 \includegraphics[width=0.6\columnwidth]{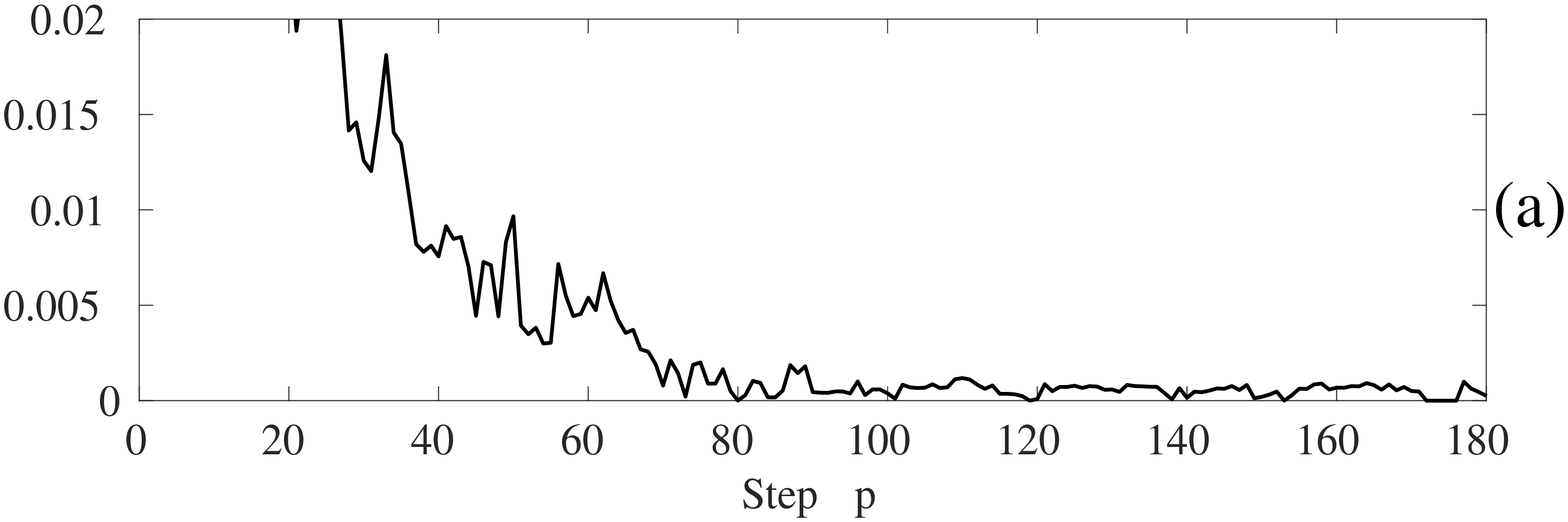} \\ 
     			 \includegraphics[width=0.6\columnwidth]{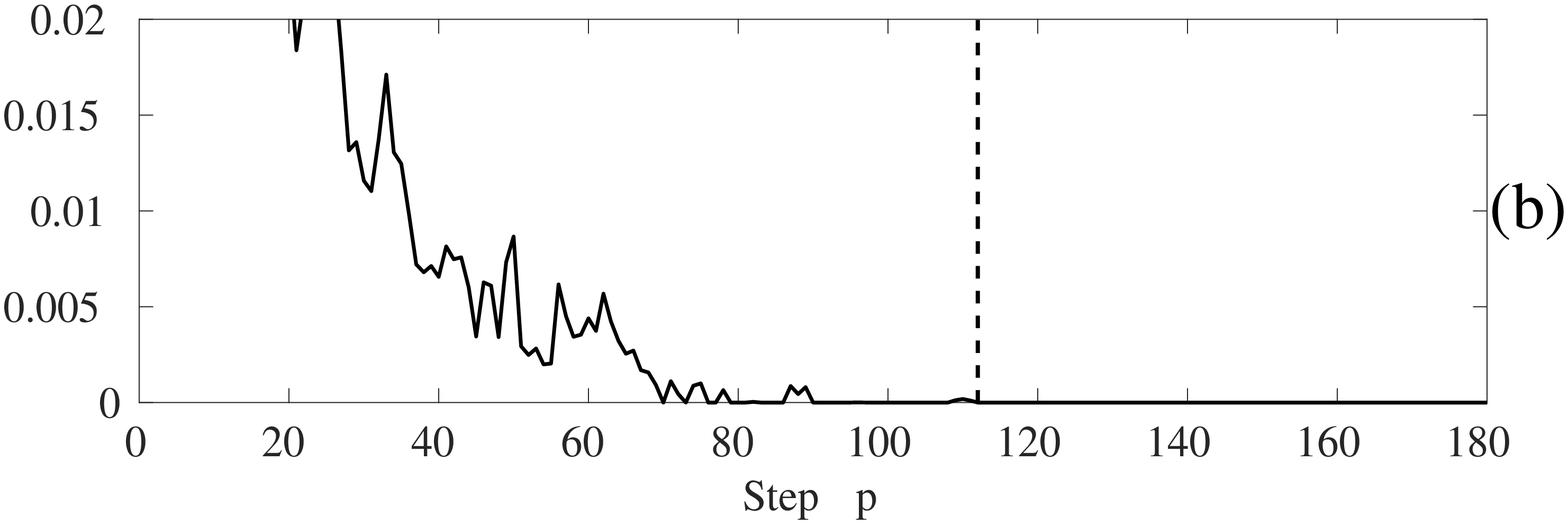} \\ 
     			 \includegraphics[width=0.6\columnwidth]{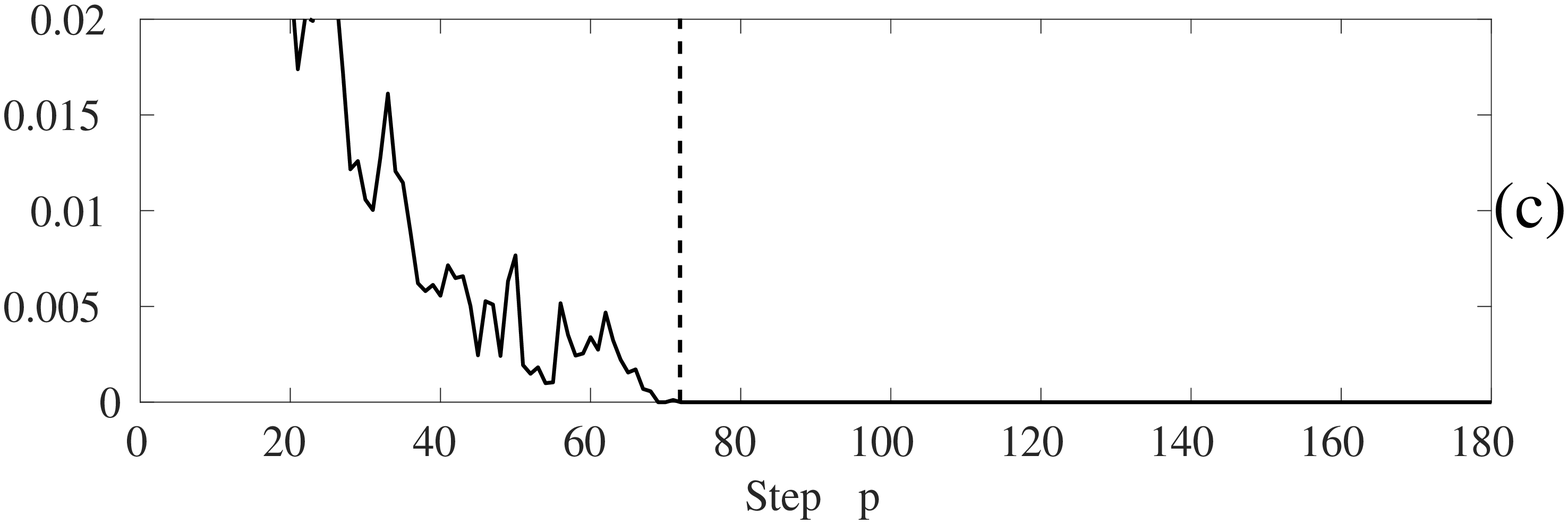}
      		\end{tabular}
      \caption{Estimated values of $\underline{\lambda}_p$ for different choice of the measurement disturbance bound; model order $o=5$. Fig. (a): $\bar{d}=0.098$; fig. (b): $\bar{d}=0.099$; fig. (c): $\bar{d}=0.1$. The dashed vertical lines indicate the value of $\bar{p}$ corresponding to each choice of $\bar{d}$.}
      \label{f:lambda_fun_d}
\end{figure} 

The first step of our identification procedure regards the estimation of the disturbance bound $\bar{d}$. Adopting the solution proposed in Procedure \ref{p:d_bar_est_procedure}, and choosing an initial model order $o=5$, we perform the calculation of $\underline{\lambda}_p$ for several values of $\bar{d}$. The result of this procedure is depicted in Fig. \ref{f:lambda_fun_d}. We decide to set $\bar{d}=0.099$, to which corresponds $\bar{p}=115$. An higher value of $\bar{d}$ would result in more conservative FPSs, while a lower value is not enough to obtain $\underline{\lambda}_p \xrightarrow{p \to \infty}0$, which is the desired result, as described by Procedure \ref{p:d_bar_est_procedure}. The obtained values of $\bar{d}$ and $\bar{p}$ are actually consistent with the true system parameters, as $\bar{d}_0=0.1$, and the time constant corresponding to the dominant poles of $G(s)$ is $T=2.5$, which results in a settling time of $125$ steps, under the chosen $T_s$.

Then, we resort to Procedure \ref{p:o_est_procedure} to obtain an estimate of the lowest order of the predictor model that verifies Assuption \ref{as:model_order}. The result of the mentioned procedure, for values of $o$ from $4$ to $2$, is shown in Fig. \ref{f:lambda_fun_o}. Here we adopt $o=3$, as it satisfies point 4) of Procedure \ref{p:o_est_procedure}; this value is consistent with the (a priori unknown) order of the considered system and verifies Assuption \ref{as:model_order}.
\begin{figure}[thpb]
	  \centering
      		\begin{tabular}{c}
     			 \includegraphics[width=0.6\columnwidth]{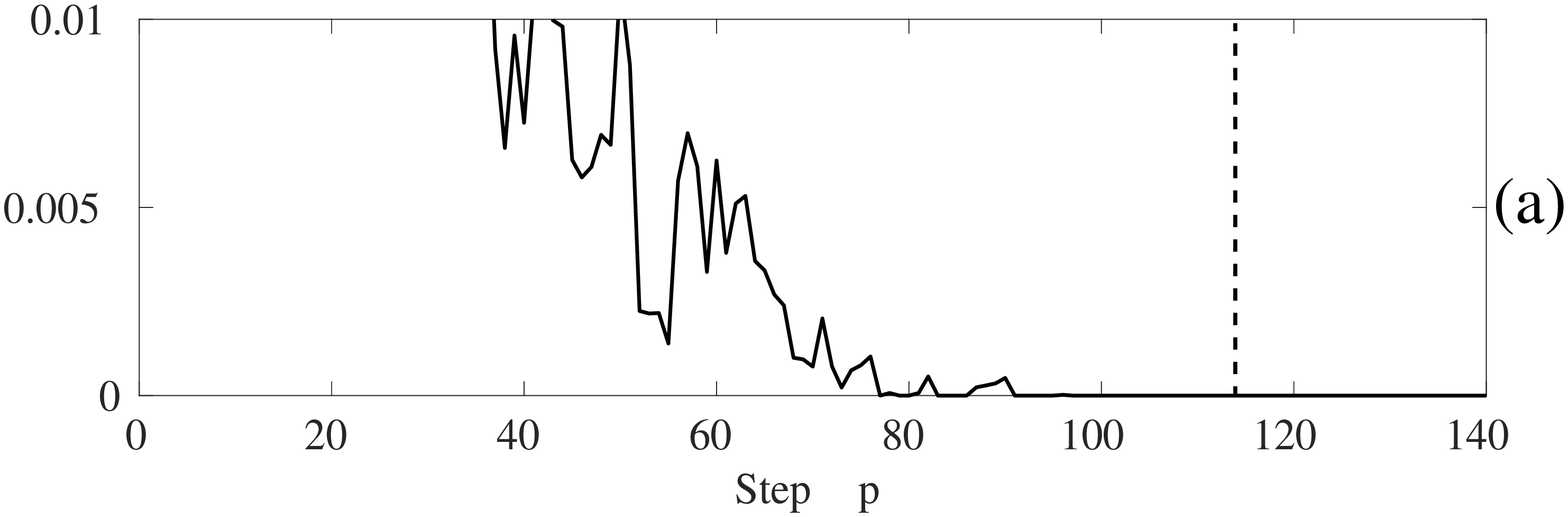} \\ 
     			 \includegraphics[width=0.6\columnwidth]{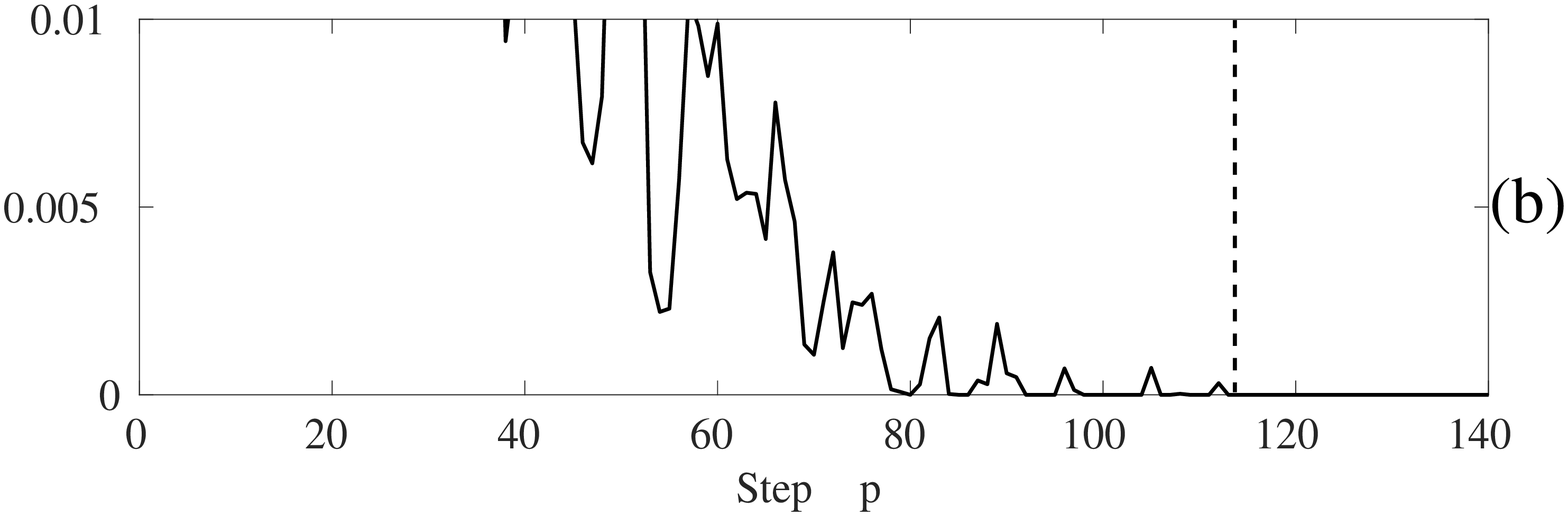} \\ 
     			 \includegraphics[width=0.6\columnwidth]{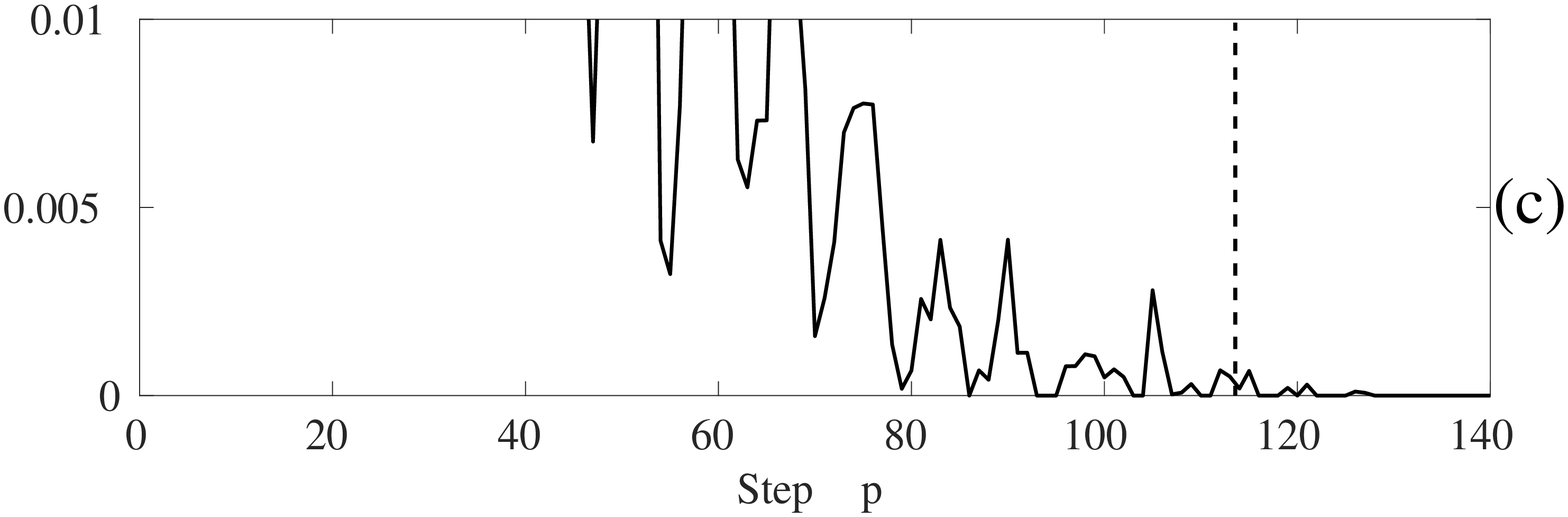}
      		\end{tabular}
      \caption{Estimated values of $\underline{\lambda}_p$ for different choice of the predictor model order; disturbance bound $\bar{d}=0.099$. Fig. (a): $o=4$; fig. (b): $o=3$; fig. (c): $o=2$. The dashed vertical lines indicate the value of $\bar{p}$ obtained for the chosen $\bar{d}$, and are used to set $o$ at the lowest possible value such that $\underline{\lambda}_p=0, \, \forall p>\bar{p}$.}
      \label{f:lambda_fun_o}
\end{figure} 

Having defined our choice of $\bar{d}$ and $o$, it is now possible to perform the procedure proposed in Section \ref{ss:dim_lambda_dbar} for the estimation of the system decay rate. Fig. \ref{f:lambda_bound_Lr} depicts the results of the estimation process of $\hat{\rho}$. Here $\hat{\rho}$ is estimated as in \eqref{eq:exp_dec_estim}; then, $\hat{L}_z$ and $\hat{L}_u$ are chosen as in \eqref{eq:L_z_estimate} and \eqref{eq:L_final_est}, respectively. This procedure results in $\hat{L}_z=1.8707$, $\hat{L}_u=0.6787$ and $\hat{\rho}=0.9645$. For a comparison, the true system decay rate is $\rho=0.96$.

Then, the values of $\underline{\lambda}_p$ corresponding to the chosen model order $o$ and disturbance bound $\bar{d}$ are inflated according to the coefficient $\alpha=1.3$, as motivated in Section \ref{s:MS_models_def}, while we set $\gamma=1.2$. The resulting $\hat{\bar{\varepsilon}}_p$ are used alongside $\bar{d}$, $\hat{L}_z$, $\hat{L}_u$ and $\hat{\rho}$, to define the FPSs for all the $p \in [1,\bar{p}]$, as in \eqref{eq:new_FPS_def}. 
\begin{figure}[thpb]
	\centering
	\includegraphics[width=0.6\columnwidth]{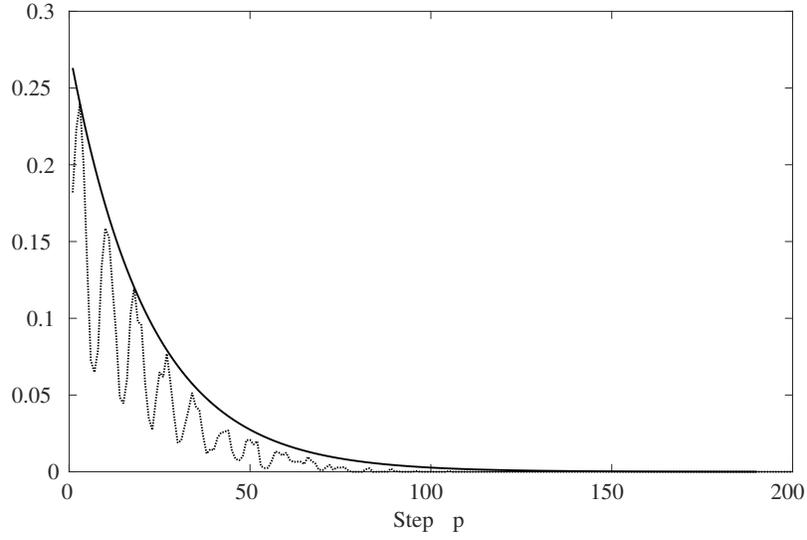}
	\caption{Exponentially decaying bound, estimated based on $\underline{\lambda}_p$. Solid line: estimated bound; dotted line: calculated values of $\hat{\bar{\varepsilon}}_p$.}
	\label{f:lambda_bound_Lr}
\end{figure} 

Finally, we adopt the identification approaches presented in Section \ref{s:one_step_estimate} to estimate the parameters of the one-step-ahead predictor, and then calculate the guaranteed accuracy bounds related to the obtained predictors, as in \eqref{eq:tau_hat_orig_def}.

As benchmarks for the proposed identification approaches, we consider a one-step-ahead prediction model identified according to the classical PEM criterion, another one identified using the SEM criterion, and the decoupled multi-step models, identified as proposed in \cite{TFFS}. Each of these decoupled multi-step models is the one that minimizes the corresponding global error bound $\hat{\tau}_p(\theta_p^*)$, and that are not linked one to the other by a one-step recursion, thus they denote the optimal performance achievable for every step $p$ in terms of minimization of the guaranteed error bound.
\begin{figure}[thpb]
	\centering
	\includegraphics[width=0.6\columnwidth]{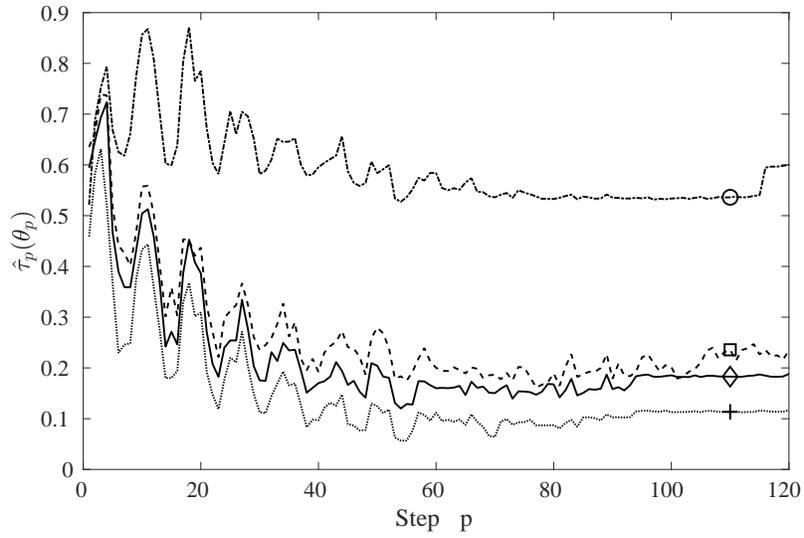}
	\caption{Guaranteed worst-case error bound. Dotted line with `$+$': multi-step approach; solid line with `$\diamond$': Method II; dashed line with `$\square$': SEM approach; dash-dot line with `$\circ$': PEM approach.}
	\label{f:Tau}
\end{figure} 
\begin{figure}[thpb]
	\centering
	\includegraphics[width=0.6\columnwidth]{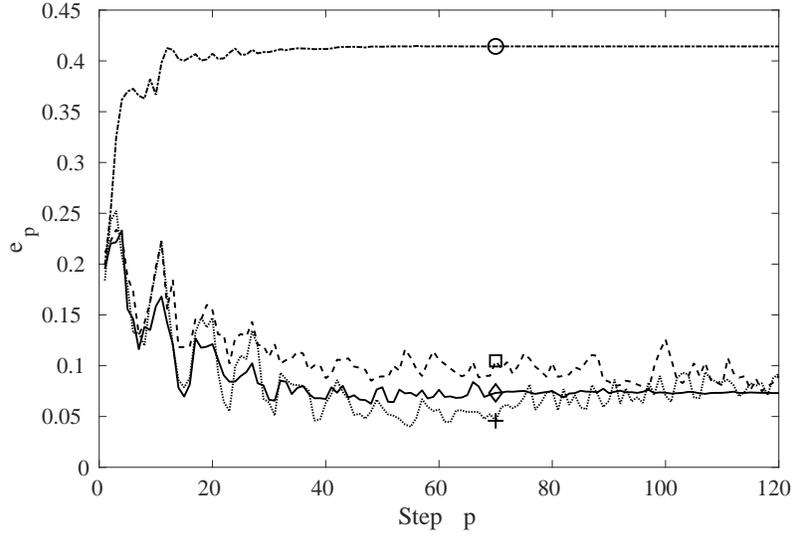}
	\caption{Estimation error calculated over validation data. Dotted line with `$+$': multi-step approach; solid line with `$\diamond$': Method II; dashed line with `$\square$': SEM approach; dash-dot line with `$\circ$': PEM approach.}
	\label{f:Tau_e_valid}
\end{figure} 
\begin{table}[thpb]
\label{t:tau_and_err}
\centering
\caption{Comparison between values of $\hat{\tau}_p$ and $e_p$ obtained by the proposed identification methods and the benchmark models.}
\begin{adjustbox}{width=0.85\columnwidth}
\begin{tabular}{|c V{3} c|c|c|c|c|c|c|c||c|c|}
\hline
 & \multicolumn{2}{c|}{PEM} & \multicolumn{2}{c|}{SEM} & \multicolumn{2}{c|}{Method I} & \multicolumn{2}{c||}{Method II} & \multicolumn{2}{c|}{Multi-step} \\ \hline
 & $\hat{\tau}_p$ & $e_{p}$ & $\hat{\tau}_p$ & $e_{p}$ & $\hat{\tau}_p$ & $e_{p}$ & $\hat{\tau}_p$ & $e_{p}$ & $\hat{\tau}_p$ & $e_{p}$ \\ \hline \hline
$p=1$ & 0.521 & 0.199 & 0.636 & 0.211 & 0.531 & \textbf{0.186} & 0.594 & 0.195 & 0.459 & 0.184 \\ \hline
$p=10$ & 0.857 & 0.367 & 0.557 & 0.197 & \textbf{0.536} & 0.163 & 0.504 & \textbf{0.158} & 0.433 & 0.193 \\ \hline
$p=35$ & 0.646 & 0.412 & 0.262 & 0.114 & 0.234 & \textbf{0.076} & 0.235 & 0.078 & 0.166 & 0.081 \\ \hline
$p=115$ & 0.540 & 0.414 & 0.227 & 0.076 & 0.185 & \textbf{0.053} & 0.187 & 0.074 & 0.116 & 0.083 \\ \hline
\end{tabular}
\end{adjustbox}
\end{table}

We use as performance indicators the guaranteed error bounds and the validation errors produced by each identification approach. The validation error for the $p$-step ahead model, calculated over the validation data-set, is defined as: 
\begin{equation}
\label{eq:valid_err}
e_p = \max_{k=o,\hdots,N_v-p} \left\vert z(k+p) - \hat{z}(k+p) \right\vert.
\end{equation}
Fig. \ref{f:Tau} and \ref{f:Tau_e_valid} depict the behavior of the guaranteed error bound and the validation error, respectively, corresponding to the various identification methods. Table I presents the values of the worst-case error bound and of the validation error of the $p$-step ahead model, for some values of $p$.

The presented numerical results show that the proposed approaches obtain better performances in terms of guaranteed error bound and validation error, with respect to both the classic PEM and SEM approaches. In particular, the second proposed identification method (Section \ref{ss:idea4}), which is based on the simulation error cost, is able to significantly improve the performance (both worst-case and actual error with validation data) of the SEM estimation approach without increasing excessively the complexity of the optimization problem.
\section{Conclusions}
\label{s:conclusions}
We presented new methods to learn one-step-ahead prediction models that provide guaranteed and minimal simulation error bounds. We resorted to the Set Membership identification framework to evaluate and optimize the worst-case simulation error, and presented new results pertaining to the estimation of noise bound, system order, and decay rate. These estimates are then employed to enforce a converging behavior also to the identified model. Finally, we proposed two possible methods to identify the model, and compared them with standard PEM and SEM approaches by means of numerical simulations. The main outcome of the presented work is that the new approaches are able to improve over standard SEM methods, in terms of both guaranteed error bounds and actual accuracy with validation data. In one of the proposed approaches, this comes with minor additional computational complexity. Future work will be devoted to prove additional theoretical properties of the proposed identification approach.
\appendix
\small
\subsection*{Proof of Theorem \ref{th:lambda_d}}
From \eqref{eq:output_def}, \eqref{eq:real_sys_ARX_model} and \eqref{eq:set_regr_noise_o}, we have that:
$$y_p=\psi_p^T \theta_p^0 + d = \left( \varphi_p^T - \Delta_p^T\right) \theta_p^0+d,$$
where $\Delta_p \in \mathbb{D}_p$, and $\mathbb{D}_p$ is defined in \eqref{eq:set_noise:o}. Then, \eqref{eq:lambda_per_proof_first_def} becomes:  
\begin{equation}
\label{eq:lambda_dis_in_proof}
\lambda_p = \min_{\theta_p \in \Omega} \; \max_{ \left[ \begin{smallmatrix} \varphi_p \\ y_p \end{smallmatrix} \right] \in \mathscr{V}_p} \left( \left\vert \varphi_p^T ( \theta_p^0 - \theta_p ) - \Delta_p^T \theta_p^0 + d \right\vert -\bar{d} \right).
\end{equation}
\subsection*{Proof of claim 1)}
\noindent Since $\bar{d} \geq 0$, we have that:
$$
\max_{ \left[ \begin{smallmatrix} \varphi_p \\ y_p \end{smallmatrix} \right] \in \mathscr{V}_p} \left( \left\vert \varphi_p^T ( \theta_p^0 - \theta_p ) - \Delta_p^T \theta_p^0 + d \right\vert -\bar{d} \right)  = \max_{ \left[ \begin{smallmatrix} \varphi_p \\ y_p \end{smallmatrix} \right] \in \mathscr{V}_p}  \left\vert \varphi_p^T ( \theta_p^0 - \theta_p ) - \Delta_p^T \theta_p^0 + d \right\vert -\bar{d} 
$$
Let us define $\Sigma_p=\varphi_p^T ( \theta_p^0 - \theta_p ) - \Delta_p^T \theta_p^0 + d$ for the sake of compactness; it is then possible to split $\left\vert \Sigma_p \right\vert$ into two terms:
$$
\lambda_p = \min_{\theta_p \in \Omega} \; \begin{cases} \underset{\left[ \begin{smallmatrix} \varphi_p \\ y_p \end{smallmatrix} \right] \in \mathscr{V}_p}{\textrm{max}}  \left( \Sigma_p \right) -\bar{d}, \; \text{if} \; \Sigma_p \geq 0 \\
\underset{\left[ \begin{smallmatrix} \varphi_p \\ y_p \end{smallmatrix} \right] \in \mathscr{V}_p}{- \textrm{min}} \left(  \Sigma_p \right) -\bar{d}, \; \text{if} \; \Sigma_p < 0
\end{cases} 
$$
Inside the set $\mathscr{V}_p$, it is always possible to find at least an occurrence of $\underline{\varphi}_p$ and $\underline{y}_p$ such that:
$$
\lambda_p=\min_{\theta_p \in \Omega} \left\vert \underline{\varphi}_p^T (\theta_p^0 - \theta_p) - \underline{\Delta}_p^T \theta_p^0 + \bar{d}_0 \right\vert - \bar{d}
$$
where $\left\vert \underline{\Delta}_p \right\vert =\left[\bar{d}_0, \, \cdots, \, \bar{d}_0, \, 0, \, \cdots, \, 0 \right]^T$. Then
$$
\begin{cases}
\underset{\left[ \begin{smallmatrix} \varphi_p \\ y_p \end{smallmatrix} \right] \in \mathscr{V}_p}{\textrm{max}} (\Sigma_p)= \underline{\varphi}_p^T ( \theta_p^0 - \theta_p ) + \bar{d}_0 \left\Vert \theta_{p,z}^0 \right\Vert_1 + \bar{d}_0, \; \text{if} \; \Sigma_p \geq 0 \\
\underset{\left[ \begin{smallmatrix} \varphi_p \\ y_p \end{smallmatrix} \right] \in \mathscr{V}_p}{\textrm{min}} (\Sigma_p)= \underline{\varphi}_p^T ( \theta_p^0 - \theta_p ) - \bar{d}_0 \left\Vert \theta_{p,z}^0 \right\Vert_1 - \bar{d}_0, \; \text{if} \; \Sigma_p < 0
\end{cases} 
$$
with $\underline{\varphi}_p^T ( \theta_p^0 - \theta_p )\geq 0$ if $\Sigma_p \geq 0$, and $\underline{\varphi}_p^T ( \theta_p^0 - \theta_p )\leq 0$ if $\Sigma_p < 0$. Then, under Assumption \ref{as:model_order}, the only optimal choice of $\theta_p$ that minimizes the resulting $\lambda_p$ is such that: 
$$\underline{\varphi}_p^T (\theta_p^0-\theta_p)=0.$$
Thus, for $\underline{\varphi}_p$, $\underline{y}_p$, $\underline{\Delta}_p$, and the corresponding optimal choice of $\theta_p$, we have: 
\begin{equation}
\label{eq:resulting_lambda_in_proof}
\lambda_p = \bar{d}_0 \left\Vert \theta_{p,z}^0 \right\Vert_1 + \bar{d}_0 - \bar{d}.
\end{equation}
Here $\bar{d}_0 \left\Vert \theta_{p,z}^0 \right\Vert_1$ represents an upper bound of the free response of the system to an initial condition given by $\underline{\Delta}_p$. For an asymptotically stable system this bound converges exponentially to zero with decay rate $\rho$, see \eqref{eq:real_decay_rate}; thus, it holds that: 
\begin{equation}
\label{eq:vanishing_Deltap}
\bar{d}_0 \left\Vert \theta_{p,z}^0 \right\Vert_1 \xrightarrow{p\to\infty} 0.
\end{equation}
Therefore, from \eqref{eq:resulting_lambda_in_proof} and \eqref{eq:vanishing_Deltap}, it follows that: 
$$\lambda_p \xrightarrow{p\to\infty} (\bar{d}_0 - \bar{d}).$$
\subsection*{Proof of claim 2)}
\noindent Let us define $\lambda_p'$ as the solution of \eqref{eq:lambda_dis_in_proof} corresponding to the previously defined $\underline{\varphi}_p$, $\underline{y}_p$, $\underline{\Delta}_p$, and $\underline{\lambda}_p'$ as the solution of \eqref{eq:lambda_est_def} corresponding to the values of $[\tilde{\varphi}_p^T \; \tilde{y}_p]^T \in \tilde{\mathscr{V}}_p^N$ under which \eqref{eq:lambda_est_def} holds with the equality. Since $\tilde{\mathscr{V}}_p^N \subset \mathscr{V}_p$, it follows that $\underline{\lambda}_p' \leq \lambda_p'$.
\subsection*{Proof of claim 3)}
\noindent Let us define
$$
\left[ \begin{smallmatrix} \bar{\varphi}_p^N \\ \bar{y}_p^N \end{smallmatrix} \right] = \text{arg} \min_{\left[ \begin{smallmatrix} \tilde{\varphi}_p \\ \tilde{y}_p \end{smallmatrix} \right] \in \tilde{\mathscr{V}}_p^N} \left\Vert \left[ \begin{smallmatrix} \underline{\varphi}_p \\ \underline{y}_p \end{smallmatrix} \right] - \left[ \begin{smallmatrix} \tilde{\varphi}_p \\ \tilde{y}_p \end{smallmatrix} \right] \right\Vert_2
$$
It follows from \eqref{eq:lambda_est_def} that
\begin{equation}
\label{eq:lambda_proof3}
\underline{\lambda}_p \geq \min_{\theta_p \in \Omega} \left( \left\vert \bar{y}_p^N - \bar{\varphi}_p^{N^T} \theta_p \right\vert - \bar{d} \right)
\end{equation}
Then, adding and subtracting $\underline{y}_p$ and $\underline{\varphi}_p^T \theta_p$ from \eqref{eq:lambda_proof3}, and neglecting the trivial case $\underline{\lambda}_p=0$, leads to:
$$
\begin{aligned}
\underline{\lambda}_p &\geq \min_{\theta_p \in \Omega} \left( \left\vert \underline{y}_p - \underline{\varphi}_p^T \theta_p - (-\bar{y}_p^N + \underline{y}_p ) + (\underline{\varphi}_p - \bar{\varphi}_p^N)^T \theta_p \right\vert -\bar{d} \right) \\
&\geq \min_{\theta_p \in \Omega} \left( \left\vert \underline{y}_p - \underline{\varphi}_p^T \theta_p \right\vert \right) - \max_{\theta_p \in \Omega}\left( \left\vert -\bar{y}_p^N + \underline{y}_p + (-\underline{\varphi}_p + \bar{\varphi}_p^N)^T \theta_p \right\vert +\bar{d} \right) \\
&= \lambda_p + \bar{d} - \max_{\theta_p \in \Omega}\left( \left\vert -\bar{y}_p^N + \underline{y}_p + (-\underline{\varphi}_p + \bar{\varphi}_p^N)^T \theta_p \right\vert \right) - \bar{d}
\end{aligned}
$$ 
Therefore:
$$
\underline{\lambda}_p \geq \lambda_p - \max_{\theta_p \in \Omega}\left( \left\vert -\bar{y}_p^N + \underline{y}_p + (-\underline{\varphi}_p + \bar{\varphi}_p^N)^T \theta_p \right\vert \right).
$$
Under Assumption \ref{as:set_distance}, we have that:
$$
\forall \beta > 0, \; \exists N < \infty : \left\Vert \bar{\varphi}_p^N - \underline{\varphi}_p \right\Vert_2 \leq \beta, \; \left\vert \bar{y}_p^N - \underline{y}_p \right\vert \leq \beta.
$$
Since $|a^T b| \leq \left\Vert a \right\Vert_2 \left\Vert b \right\Vert_2$, we have:
$$
\begin{aligned}
\underline{\lambda}_p &\geq \lambda_p - \max_{\theta_p \in \Omega} \left( \left\vert \bar{y}_p^N - \underline{y}_p \right\vert + \left\Vert \underline{\varphi}_p - \bar{\varphi}_p^N \right\Vert_2 \cdot \left\Vert \theta_p \right\Vert_2 \right) \\
&\geq \lambda_p - \beta \left(1+ \max_{\theta_p\in \Theta_p^0} \left\Vert \theta_p \right\Vert_2 \right)
\end{aligned}
$$
Where $\Theta_p^0$ is the set of parameter such that:
$$ \Theta_p^0 = \left\{ \theta_p^0 : \theta_p^0 = \arg \min_{\theta_p \in \Omega} \lambda_p(\theta_p) \right\}. $$
Then, claim 3) of Theorem \ref{th:lambda_d} is verified by choosing 
$$ \beta \leq \frac{\eta}{\left(1 + \underset{\theta_p\in \Theta_p^0}{\textrm{max}} \left\Vert \theta_p \right\Vert_2 \right)}. $$

\subsection*{Proof of Corollary \ref{th:Delta_theta}}
A straightforward consequence of Theorem \ref{th:lambda_d} is that, if $\bar{d}=\bar{d}_0$, then $\lambda_p \xrightarrow{p \to \infty} 0$. In addition, following the procedure adopted for the proof of claim 1) of Theorem \ref{th:lambda_d}, we can say that it is always possible to find at least an occurrence of $\underline{\varphi}_p$ and $\underline{y}_p$ inside the set $\mathscr{V}_p$, such that \eqref{eq:lambda_dis_in_proof} reduces to: 
\begin{equation}
\label{eq:resulting_lambda_in_proof2}
\lambda_p =  \bar{d}_0 \left\Vert \theta_{p,z}^0 \right\Vert_1.
\end{equation}
From \eqref{eq:real_decay_rate} it follows that $\left\vert \theta_{p,z}^{0,(i)} \right\vert \leq L_z \rho^{p+1}, \; i=1,\hdots,n-1$. Thus, \eqref{eq:resulting_lambda_in_proof2} becomes:
\begin{equation}
\label{eq:resulting_lambda_in_proof3}
\lambda_p =  \bar{d}_0 \left\Vert \theta_{p,z}^0 \right\Vert_1 \leq n \, \bar{d}_0 L_z \rho^{p+1}.
\end{equation}
\normalsize

\bibliographystyle{IEEEtran}

\end{document}